\documentclass[12pt]{article}

\usepackage[a4paper]{geometry}
\usepackage{amssymb,amsmath,amsthm,bm,color,scalerel}
\usepackage{hyperref}

\usepackage{tikz}
\usetikzlibrary{svg.path}
\definecolor{orcid_color}{HTML}{A6CE39}

\hypersetup{pdfborder={0 0 0}} 
\DeclareRobustCommand{\orcidicon}{%
	\raisebox{.3mm}{\scalerel*{%
	\begin{tikzpicture}[xscale=1,yscale=-1,transform shape]
	\filldraw[color=orcid_color] svg {M256,128c0,70.7-57.3,128-128,128C57.3,256,0,198.7,0,128C0,57.3,57.3,0,128,0C198.7,0,256,57.3,256,128z};
	\filldraw[color=white] svg {M86.3,186.2H70.9V79.1h15.4v48.4V186.2z} svg {M108.9,79.1h41.6c39.6,0,57,28.3,57,53.6c0,27.5-21.5,53.6-56.8,53.6h-41.8V79.1z M124.3,172.4h24.5
		c34.9,0,42.9-26.5,42.9-39.7c0-21.5-13.7-39.7-43.7-39.7h-23.7V172.4z} svg {M88.7,56.8c0,5.5-4.5,10.1-10.1,10.1c-5.6,0-10.1-4.6-10.1-10.1c0-5.6,4.5-10.1,10.1-10.1
		C84.2,46.7,88.7,51.3,88.7,56.8z};
	\end{tikzpicture}}{|}}%
}
\newcommand{\orcid}[1]{\href{https://orcid.org/#1}{\orcidicon}}

\theoremstyle{plain}
\newtheorem{theorem}{Theorem}[section]
\newtheorem{proposition}[theorem]{Proposition}
\newtheorem{lemma}[theorem]{Lemma}
\theoremstyle{remark}
\newtheorem{remark}[theorem]{Remark}

\newcommand{\arxiv}[1]{arXiv:\href{http://arxiv.org/abs/#1}{#1}}

\newcommand{\gauss}[2]{\genfrac{[}{]}{0pt}{}{#1}{#2}}

\title{\Large\bfseries The Terwilliger algebra of \\ the twisted Grassmann graph: the thin case}
\author{Hajime Tanaka\,\orcid{0000-0002-5958-0375} \qquad Tao Wang\footnote{Corresponding author} \\
\small Research Center for Pure and Applied Mathematics\\[-0.8ex]
\small Graduate School of Information Sciences\\[-0.8ex]
\small Tohoku University\\[-0.8ex]
\small Sendai, Japan\\
\small\tt htanaka@tohoku.ac.jp \qquad wang.tao.t7@dc.tohoku.ac.jp}
\date{}

\begin{document}

\maketitle

\hypersetup{pdfborder={0 0 1}} 

\begin{abstract}
The \emph{Terwilliger algebra} $T(x)$ of a finite connected simple graph $\Gamma$ with respect to a vertex $x$ is the complex semisimple matrix algebra generated by the adjacency matrix $A$ of $\Gamma$ and the diagonal matrices $E_i^*(x)=\operatorname{diag}(v_i)$ $(i=0,1,2,\dots)$, where $v_i$ denotes the characteristic vector of the set of vertices at distance $i$ from $x$.
The \emph{twisted Grassmann graph} $\tilde{J}_q(2D+1,D)$ discovered by Van Dam and Koolen in 2005 has two orbits of the automorphism group on its vertex set, and it is known that one of the orbits has the property that $T(x)$ is \emph{thin} whenever $x$ is chosen from it, i.e., every irreducible $T(x)$-module $W$ satisfies $\dim E_i^*(x)W\leqslant 1$ for all $i$.
In this paper, we determine all the irreducible $T(x)$-modules of $\tilde{J}_q(2D+1,D)$ for this ``thin'' case.
\end{abstract}

\section{Introduction}

In 2005, Van Dam and Koolen \cite{DK2005IM} discovered a new infinite family of distance-regular graphs with unbounded diameter, which they call the \emph{twisted Grassmann graphs}.
Let $q$ be a prime power, and let $D$ be an integer at least two.
Fix a hyperplane $H$ of the vector space $\mathbb{F}_q^{2D+1}$ over the finite field $\mathbb{F}_q$.
Let $X'$ be the set of $(D+1)$-dimensional subspaces of $\mathbb{F}_q^{2D+1}$ not contained in $H$, and let $X''$ be the set of $(D-1)$-dimensional subspaces of $H$.
The twisted Grassmann graph $\tilde{J}_q(2D+1,D)$ has vertex set $X=X'\sqcup X''$, where two vertices $y$ and $z$ are adjacent whenever
\begin{equation}\label{adjacency}
	\dim y+\dim z-2\dim y\cap z=2.
\end{equation}
The graph $\tilde{J}_q(2D+1,D)$ has the same intersection array as the Grassmann graph $J_q(2D+1,D)$ on the set of $D$-dimensional subspaces of $\mathbb{F}_q^{2D+1}$.
A particularly interesting feature of the twisted Grassmann graphs is that these are not vertex-transitive.
All the previously known families of distance-regular graphs with unbounded diameter are at least vertex-transitive, and many of them are in fact distance-transitive.
For more information on the twisted Grassmann graphs, see \cite{BFK2009EJC,DK2005IM,DKT2016EJC,FKT2006IIG,LHW2017GC,Munemasa2017DCC,MT2011IIG,Tanaka2011EJC,Tanaka2012C}.

In this paper, we discuss the \emph{Terwilliger algebra} of the twisted Grassmann graph $\tilde{J}_q(2D+1,D)$.
In general, for a finite connected simple graph $\Gamma$ with vertex set $V\Gamma$, the Terwilliger algebra $T(x)$ of $\Gamma$ with respect to the \emph{base vertex} $x\in V\Gamma$ is the subalgebra of the $\mathbb{C}$-algebra of complex matrices with rows and columns indexed by $V\Gamma$, generated by the adjacency matrix $A$ of $\Gamma$ and the diagonal matrices $E_i^*(x)=\operatorname{diag}(v_i)$ $(i=0,1,2,\dots)$, where $v_i$ denotes the characteristic vector of the set of vertices at distance $i$ from $x$ \cite{Terwilliger1992JAC,Terwilliger1993JACa,Terwilliger1993JACb,Terwilliger1993N}.
The algebra $T(x)$ is non-commutative whenever $|V\Gamma|>1$, and is semisimple as it is closed under conjugate-transpose.
For the role of the Terwilliger algebra in the study of distance-regular graphs, see \cite{DKT2016EJC} and the references therein.
We note that, for the examples of distance-regular graphs with large diameter known before 2005, the Terwilliger algebra $T(x)$ is independent of the choice of the base vertex $x$ up to isomorphism, since these graphs are vertex-transitive.
On the other hand, the automorphism group of $\tilde{J}_q(2D+1,D)$ has two orbits on the vertex set $X$, namely $X'$ and $X''$, and $T(x)$ indeed depends on the orbit to which $x$ belongs.
More specifically, Bang, Fujisaki, and Koolen \cite{BFK2009EJC} showed that $T(x)$ is thin for $x\in X''$ and is non-thin for $x\in X'$.
Here, $T(x)$ is called \emph{thin} if every irreducible $T(x)$-module $W$ satisfies $\dim E_i^*(x)W\leqslant 1$ for all $i$, and \emph{non-thin} otherwise.

The Terwilliger algebra of the Grassmann graph $J_q(n,D)$ $(n\geqslant 2D)$ is known to be thin, and its irreducible modules have been described; see, e.g., \cite{LIW2020LAA,TTW2021+pre,Terwilliger1993JACb,Watanabe2015M}.
The goal of this paper is to determine the irreducible modules of the Terwilliger algebra $T(x)$ of $\tilde{J}_q(2D+1,D)$ for the thin case, i.e., $x\in X''$.
While the non-thin case for $\tilde{J}_q(2D+1,D)$ currently appears to be beyond our scope, the thin case is already more involved than $J_q(n,D)$.
Our approach is to embed $T(x)$ into a larger matrix algebra, denoted by $\tilde{\mathcal{H}}$, whose rows and columns are indexed by the subspaces of $\mathbb{F}_q^{2D+1}$, not just $X$.
The algebra $\tilde{\mathcal{H}}$ is generated by three types of ``lowering'' matrices and ``raising'' matrices, together with certain $0$-$1$ diagonal matrices.
We note that $\tilde{\mathcal{H}}$ is an extension of the incidence algebra of the subspace lattice of $\mathbb{F}_q^{2D+1}$ (cf.~\cite{Terwilliger1990P}), and that it is also a subalgebra of the centralizer algebra of the parabolic subgroup of $\operatorname{GL}(\mathbb{F}_q^{2D+1})$ which stabilizes both $H$ and $x\,(\subset H)$, acting on the subspaces of $\mathbb{F}_q^{2D+1}$.
A similar approach was previously taken for $J_q(n,D)$, and the representation theory of the corresponding algebra, denoted by $\mathcal{H}$, has been fully developed by Watanabe \cite{Watanabe2017JA} (and also \cite{TTW2021+pre}).
In Section \ref{sec: H}, we collect necessary results from \cite{TTW2021+pre,Watanabe2017JA}.
In Section \ref{sec: extension of H}, after introducing the algebra $\tilde{\mathcal{H}}$, we extend some of the results of Srinivasan \cite{Srinivasan2014EJC}, where he gave, among other results, a linear algebraic interpretation of the Goldman--Rota identity for the number of subspaces of $\mathbb{F}_q^n$.
This enables us to relate the representation theory of $\tilde{\mathcal{H}}$ to that of $\mathcal{H}$, which then leads to a complete description (Theorem \ref{irreducible tilde-H-modules}) of the irreducible $\tilde{\mathcal{H}}$-modules.
Section \ref{sec: irreducible T-modules} is devoted to finding all the irreducible $T(x)$-modules.
The algebra $\tilde{\mathcal{H}}$ contains the $0$-$1$ diagonal matrix $\tilde{\mathcal{E}}^*$ whose $(y,y)$-entry equals one if and only if $y\in X$.
We first observe that $T(x)$ is viewed as a subalgebra of the algebra $\tilde{\mathcal{E}}^*\tilde{\mathcal{H}}\tilde{\mathcal{E}}^*$ with identity $\tilde{\mathcal{E}}^*$.
Put differently, $T(x)$ is contained in the set of principal submatrices indexed by $X$, of the elements of $\tilde{\mathcal{H}}$.
We then show that, for every irreducible $\tilde{\mathcal{H}}$-module $\tilde{\mathcal{W}}$, the subspace $\tilde{\mathcal{E}}^*\tilde{\mathcal{W}}$ is either an irreducible $T(x)$-module, or the direct sum of two irreducible $T(x)$-modules.
We also find the isomorphisms among these irreducible $T(x)$-modules.
In view of the semisimplicity of $T(x)$, this completes the classification of the irreducible $T(x)$-modules.
Our main results are Theorems \ref{irreducible T-modules of type 1}, \ref{irreducible T-modules of type 2}, and \ref{irreducible T-modules of type 3}.
See also the comments after Theorem \ref{irreducible T-modules of type 3}.

Throughout this paper, we fix a prime power $q$ and use the following notation:
\begin{gather*}
	(\alpha)_i=(\alpha;q)_i=\prod_{\ell=0}^{i-1}(1-\alpha q^{\ell}) \qquad (\alpha\in\mathbb{C},\,i\in\mathbb{Z}_{\geqslant 0}), \\
	\gauss{m}{n}=\gauss{m}{n}_q=\frac{(q)_m}{(q)_n(q)_{m-n}} \qquad (m,n\in\mathbb{Z}_{\geqslant 0},\, m\geqslant n).
\end{gather*}
For every non-empty finite set $S$, we let $\operatorname{Mat}_S(\mathbb{C})$ denote the $\mathbb{C}$-algebra of complex matrices with rows and columns indexed by $S$, and we also let $\mathbb{C}S$ denote the $\mathbb{C}$-vector space with basis $S$, on which $\operatorname{Mat}_S(\mathbb{C})$ acts from the left in the standard manner.

\section{The algebra \texorpdfstring{$\mathcal{H}$}{H}}
\label{sec: H}

Let $a$ and $b$ be non-negative integers, and let $P$ be the set of all subspaces of $\mathbb{F}_q^{a+b}$.
We will always fix $x\in P$ with $\dim x=a$.
For $0\leqslant i\leqslant a$ and $0\leqslant j\leqslant b$, let
\begin{equation*}
	P_{i,j}=\{y\in P:\dim x\cap y=i,\,\dim y=i+j\}.
\end{equation*}
We note that the $P_{i,j}$ give a partition of $P$, and that (cf.~\cite[Lemma 9.3.2]{BCN1989B})
\begin{equation*}
	|P_{i,j}|=q^{(a-i)j}\gauss{a}{i}\gauss{b}{j}.
\end{equation*}
In particular, if we let $s_q(n)$ denote the number of subspaces of $\mathbb{F}_q^n$, then we have
\begin{equation}\label{number of subspaces}
	s_q(a+b)=|P|=\sum_{i=0}^a\sum_{j=0}^bq^{(a-i)j}\gauss{a}{i}\gauss{b}{j}.
\end{equation}
For convenience, we set $P_{i,j}:=\emptyset$ for $i,j\in\mathbb{Z}$ unless $0\leqslant i\leqslant a$ and $0\leqslant j\leqslant b$.

For $0\leqslant i\leqslant a$ and $0\leqslant j\leqslant b$, let $\mathcal{E}_{i,j}^*\in \operatorname{Mat}_P(\mathbb{C})$ be the diagonal matrix with $(y,y)$-entry
\begin{equation*}
	(\mathcal{E}_{i,j}^*)_{y,y}=\begin{cases} 1 & \text{if} \ y\in P_{i,j}, \\ 0 & \text{otherwise}, \end{cases} \qquad (y\in P).
\end{equation*}
Moreover, let $\mathcal{L}_1,\mathcal{L}_2,\mathcal{R}_1,\mathcal{R}_2\in \operatorname{Mat}_P(\mathbb{C})$ be such that, for $y,z\in P$ with $z\in P_{i,j}$,
\begin{align*}
	(\mathcal{L}_1)_{y,z}&=\begin{cases} 1 & \text{if} \ y\in P_{i-1,j}, \ y\subset z, \\ 0 & \text{otherwise}, \end{cases} & (\mathcal{R}_1)_{y,z}&=\begin{cases} 1 & \text{if} \ y\in P_{i+1,j}, \ y\supset z, \\ 0 & \text{otherwise}, \end{cases} \\
	(\mathcal{L}_2)_{y,z}&=\begin{cases} 1 & \text{if} \ y\in P_{i,j-1}, \ y\subset z, \\ 0 & \text{otherwise}, \end{cases} & (\mathcal{R}_2)_{y,z}&=\begin{cases} 1 & \text{if} \ y\in P_{i,j+1}, \ y\supset z, \\ 0 & \text{otherwise}. \end{cases}
\end{align*}
We note that $(\mathcal{L}_1)^{\mathsf{T}}=\mathcal{R}_1$ and $(\mathcal{L}_2)^{\mathsf{T}}=\mathcal{R}_2$, where ${}^{\mathsf{T}}$ denotes transpose.
Let $\mathcal{H}$ be the subalgebra of $\operatorname{Mat}_P(\mathbb{C})$ generated by $\mathcal{L}_1,\mathcal{L}_2,\mathcal{R}_1,\mathcal{R}_2$, and all the $\mathcal{E}_{i,j}^*$.
The algebra $\mathcal{H}$ is semisimple as it is closed under conjugate-transpose.
We note that every irreducible $\mathcal{H}$-module appears in $\mathbb{C}P$ up to isomorphism.

Let $\mathcal{W}$ be an irreducible $\mathcal{H}$-module.
Let
\begin{align*}
	\nu &= \min\{i:(\mathcal{E}_{i,0}^*+\dots+\mathcal{E}_{i,b}^*)\mathcal{W}\ne 0\}, \\
	\nu' &= \max\{i:(\mathcal{E}_{i,0}^*+\dots+\mathcal{E}_{i,b}^*)\mathcal{W}\ne 0\}, \\
	\mu &= \min\{j:(\mathcal{E}_{0,j}^*+\dots+\mathcal{E}_{a,j}^*)\mathcal{W}\ne 0\}, \\
	\mu' &= \max\{j:(\mathcal{E}_{0,j}^*+\dots+\mathcal{E}_{a,j}^*)\mathcal{W}\ne 0\}.
\end{align*}
We call $(\nu,\mu)$ (resp.~$(\nu',\mu')$) the \emph{lower endpoint} (resp.~\emph{upper endpoint}) of $\mathcal{W}$.
It is known (\cite[Lemma 8.3]{Watanabe2017JA}) that
\begin{equation*}
	\nu+\nu'+\mu+\mu'=a+b.
\end{equation*}
We let
\begin{equation*}
	\rho=a-\nu-\nu'=\mu+\mu'-b,
\end{equation*}
and call $\rho$ the \emph{index} of $\mathcal{W}$.

\begin{theorem}[{\cite{TTW2021+pre,Watanabe2017JA}}]\label{irreducible H-modules}
There exists an irreducible $\mathcal{H}$-module with lower endpoint $(\nu,\mu)$ and index $\rho$ if and only if
\begin{equation*}
	0\leqslant \nu\leqslant a, \qquad 0\leqslant\mu\leqslant b, \qquad \max\{0,2\mu-b\}\leqslant \rho\leqslant \min\{a-2\nu,\mu\}.
\end{equation*}
Let $\mathcal{W}$ be an irreducible $\mathcal{H}$-module with lower endpoint $(\nu,\mu)$ and index $\rho$.
Then $\mathcal{W}$ has a basis $w_{i,j}$ $(\nu\leqslant i\leqslant a-\nu-\rho, \ \mu\leqslant j\leqslant b-\mu+\rho)$ such that $w_{i,j}\in\mathcal{E}_{i,j}^*\mathcal{W}$, and
\begin{align*}
	\mathcal{L}_1w_{i,j}&=q^{\nu+j}\gauss{a-\nu-\rho-i+1}{1}\gauss{i-\nu}{1}w_{i-1,j},  & \mathcal{R}_1w_{i,j}&=w_{i+1,j}, \\
	\mathcal{L}_2w_{i,j}&=q^{a-\nu+\mu-\rho}\gauss{b-\mu+\rho-j+1}{1}\gauss{j-\mu}{1}w_{i,j-1}, & \mathcal{R}_2w_{i,j}&=q^{\nu-i}w_{i,j+1}
\end{align*}
for all $i,j$, where we set $w_{\nu-1,j}=w_{a-\nu-\rho+1,j}=w_{i,\mu-1}=w_{i,b-\mu+\rho+1}:=0$.
In particular, the isomorphism class of $\mathcal{W}$ is determined by $\nu,\mu$, and $\rho$.
Moreover, the multiplicity $m_{\nu,\mu,\rho}$ of $\mathcal{W}$ in $\mathbb{C}P$ is given by
\begin{equation*}
	m_{\nu,\mu,\rho}=\frac{ (-1)^{\rho} (q)_a (q)_b (1-q^{a-2\nu-\rho+1}) (1-q^{b-2\mu+\rho+1}) q^{\nu+\mu-\rho+\binom{\rho}{2}} }{ (q)_{a-\nu-\rho+1} (q)_{b-\mu+1} (q)_{\nu} (q)_{\mu-\rho} (q)_{\rho}}.
\end{equation*}
\end{theorem}

\begin{remark}
The generators of the algebra $\mathcal{H}$ commute with the action on $P$ of the maximal parabolic subgroup of $\mathrm{GL}(\mathbb{F}_q^{a+b})$ stabilizing the subspace $x$.
Hence $\mathcal{H}$ is a subalgebra of the centralizer algebra of this parabolic subgroup.
However, by comparing Theorem \ref{irreducible H-modules} with the results of Dunkl \cite{Dunkl1978MM} we can show that the two algebras are in fact equal.
In particular, the formula for the multiplicity $m_{\nu,\mu,\rho}$ above agrees with \cite[Proposition 4.15]{Dunkl1978MM} under the replacement $(\nu,\mu,\rho)\mapsto(m,n+r,r)$.
See also \cite{Watanabe2015M}.
\end{remark}

\section{The algebra \texorpdfstring{$\tilde{\mathcal{H}}$}{tilde H}}
\label{sec: extension of H}

Let $a$ and $b$ be as in the previous section.
Let $\tilde{P}$ be the set of all subspaces of $\mathbb{F}_q^{a+b+1}$.
We will again fix $x\in \tilde{P}$ with $\dim x=a$, and we will also fix a hyperplane $H\in\tilde{P}$ containing $x$.
For $0\leqslant i\leqslant a$, $0\leqslant j\leqslant b$, and $0\leqslant k\leqslant 1$, let
\begin{equation*}
	\tilde{P}_{i,j,k}=\{y\in \tilde{P}:\dim x\cap y=i,\,\dim H\cap y=i+j,\,\dim y=i+j+k\}.
\end{equation*}
Then the $\tilde{P}_{i,j,k}$ give a partition of $\tilde{P}$, and
\begin{equation*}
	|\tilde{P}_{i,j,k}|=q^{(a-i)j+(a+b-i-j)k}\gauss{a}{i}\gauss{b}{j}.
\end{equation*}
For convenience, we set $\tilde{P}_{i,j,k}:=\emptyset$ for $i,j,k\in\mathbb{Z}$ unless $0\leqslant i\leqslant a$, $0\leqslant j\leqslant b$, and $0\leqslant k\leqslant 1$.

For $0\leqslant i\leqslant a$, $0\leqslant j\leqslant b$, and $0\leqslant k\leqslant 1$, let $\tilde{\mathcal{E}}_{i,j,k}^*\in \operatorname{Mat}_{\tilde{P}}(\mathbb{C})$ be the diagonal matrix with $(y,y)$-entry
\begin{equation*}
	(\tilde{\mathcal{E}}_{i,j,k}^*)_{y,y}=\begin{cases} 1 & \text{if} \ y\in \tilde{P}_{i,j,k}, \\ 0 & \text{otherwise}, \end{cases} \qquad (y\in \tilde{P}).
\end{equation*}
Moreover, let $\tilde{\mathcal{L}}_1,\tilde{\mathcal{L}}_2,\tilde{\mathcal{L}}_3,\tilde{\mathcal{R}}_1,\tilde{\mathcal{R}}_2,\tilde{\mathcal{R}}_3\in \operatorname{Mat}_{\tilde{P}}(\mathbb{C})$ be such that, for $y,z\in \tilde{P}$ with $z\in \tilde{P}_{i,j,k}$,
\begin{align*}
	(\tilde{\mathcal{L}}_1)_{y,z}&=\begin{cases} 1 & \text{if} \ y\in \tilde{P}_{i-1,j,k}, \ y\subset z, \\ 0 & \text{otherwise}, \end{cases} & (\tilde{\mathcal{R}}_1)_{y,z}&=\begin{cases} 1 & \text{if} \ y\in \tilde{P}_{i+1,j,k}, \ y\supset z, \\ 0 & \text{otherwise}, \end{cases} \\
	(\tilde{\mathcal{L}}_2)_{y,z}&=\begin{cases} 1 & \text{if} \ y\in \tilde{P}_{i,j-1,k}, \ y\subset z, \\ 0 & \text{otherwise}, \end{cases} & (\tilde{\mathcal{R}}_2)_{y,z}&=\begin{cases} 1 & \text{if} \ y\in \tilde{P}_{i,j+1,k}, \ y\supset z, \\ 0 & \text{otherwise}, \end{cases} \\
	(\tilde{\mathcal{L}}_3)_{y,z}&=\begin{cases} 1 & \text{if} \ y\in \tilde{P}_{i,j,k-1}, \ y\subset z, \\ 0 & \text{otherwise}, \end{cases} & (\tilde{\mathcal{R}}_3)_{y,z}&=\begin{cases} 1 & \text{if} \ y\in \tilde{P}_{i,j,k+1}, \ y\supset z, \\ 0 & \text{otherwise}. \end{cases}
\end{align*}
We note that $(\tilde{\mathcal{L}}_1)^{\mathsf{T}}=\tilde{\mathcal{R}}_1$, $(\tilde{\mathcal{L}}_2)^{\mathsf{T}}=\tilde{\mathcal{R}}_2$, and $(\tilde{\mathcal{L}}_3)^{\mathsf{T}}=\tilde{\mathcal{R}}_3$.
Let $\tilde{\mathcal{H}}$ be the subalgebra of $\operatorname{Mat}_{\tilde{P}}(\mathbb{C})$ generated by $\tilde{\mathcal{L}}_1,\tilde{\mathcal{L}}_2,\tilde{\mathcal{L}}_3,\tilde{\mathcal{R}}_1,\tilde{\mathcal{R}}_2,\tilde{\mathcal{R}}_3$, and all the $\tilde{\mathcal{E}}_{i,j,k}^*$.
The algebra $\tilde{\mathcal{H}}$ is semisimple as it is closed under conjugate-transpose.
We note that every irreducible $\tilde{\mathcal{H}}$-module appears in $\mathbb{C}\tilde{P}$ up to isomorphism.

Our goal in this section is to describe the irreducible $\tilde{\mathcal{H}}$-modules.
To this end, we extend some of the results of Srinivasan \cite{Srinivasan2014EJC}.
Let $G$ be the subgroup of $\mathrm{SL}(\mathbb{F}_q^{a+b+1})$ consisting of the elements which fix every vector in $H$.
If we fix a basis $\mathbf{u}_1,\dots,\mathbf{u}_{a+b+1}$ of $\mathbb{F}_q^{a+b+1}$ such that $H=\operatorname{span}\{\mathbf{u}_1,\dots,\mathbf{u}_{a+b}\}$, then the matrices representing the elements of $G$ with respect to this basis are of the form
\begin{equation}\label{matrix form}
	\begin{bmatrix} 1 &&&& \!\!\!\!\alpha_1 \\[-.8mm] & \!\!\!\!\cdot &&& \!\!\!\!\!\cdot \\[-.8mm] && \!\!\!\!\!\cdot && \!\!\!\!\!\cdot \\[-.8mm] &&& \!\!\!\!\!1 & \!\!\!\!\alpha_{a+b} \\ &&&& \!\!\!\!\!1 \end{bmatrix} \qquad (\alpha_1,\dots,\alpha_{a+b}\in\mathbb{F}_q).
\end{equation}
Observe that $G$ is abelian and is isomorphic to the additive group $\mathbb{F}_q^{a+b}$, and that the $\tilde{P}_{i,j,k}$ are $G$-invariant.
Let
\begin{equation*}
	\tilde{P}_0=\bigsqcup_{i=0}^a\bigsqcup_{j=0}^b\tilde{P}_{i,j,0}, \qquad \tilde{P}_1=\bigsqcup_{i=0}^a\bigsqcup_{j=0}^b\tilde{P}_{i,j,1}.
\end{equation*}
Thus, $\tilde{P}_0$ is the set of subspaces of $H$, and $\tilde{P}_1$ is the set of subspaces of $\mathbb{F}_q^{a+b+1}$ not contained in $H$.
We define the equivalence relation $\sim$ on $\tilde{P}_1$ by
\begin{equation*}
	y\sim z \quad \text{if and only if} \quad H\cap y=H\cap z  \qquad (y,z\in\tilde{P}_1).
\end{equation*}
We observe that the equivalence classes of this relation are precisely the $G$-orbits on $\tilde{P}_1$.
For every $y\in \tilde{P}_1$, let $G_y$ denote the stabilizer of $y$ in $G$.

Let $\widehat{G}$ be the character group of $G$ with trivial character $1_G$.
For $0\leqslant i\leqslant a$ and $0\leqslant j\leqslant b$, let $\psi_{i,j}$ be the permutation character of $G$ on $\tilde{P}_{i,j,1}$.
Note that $(g-\mathrm{id})(\mathbb{F}_q^{a+b+1})$ is a one-dimensional subspace of $H$ for every $g\in G$ with $g\ne \operatorname{id}$ (cf.~\eqref{matrix form}).
The elements of $G$ such that $(g-\operatorname{id})(\mathbb{F}_q^{a+b+1})\subset x$ form a subgroup of $G$ of order $q^a$, which we denote by $K$.
The following extends \cite[Theorem 2.3]{Srinivasan2014EJC}, and the proof is straightforward.

\begin{lemma}\label{characters}
The following hold:
\begin{enumerate}
\item For $0\leqslant i\leqslant a$, $0\leqslant j\leqslant b$, and $g\in G$, we have
\begin{equation*}
	\psi_{i,j}(g)=\begin{cases} q^{(a-i)j+(a+b-i-j)}\gauss{a}{i}\gauss{b}{j} & \text{if} \ g=\operatorname{id}, \\ q^{(a-i)j+(a+b-i-j)}\gauss{a-1}{i-1}\gauss{b}{j} & \text{if} \ g\ne \operatorname{id}, \ g\in K, \\ q^{(a-i)(j-1)+(a+b-i-j)}\gauss{a}{i}\gauss{b-1}{j-1} & \text{if} \ g\in G\setminus K. \end{cases}
\end{equation*}
\item For $0\leqslant i\leqslant a$, $0\leqslant j\leqslant b$, and $\chi\in \widehat{G}$, we have
\begin{equation*}
	[\chi,\psi_{i,j}]=\begin{cases} q^{(a-i)j}\gauss{a}{i}\gauss{b}{j} & \text{if} \ \chi=1_G,   \\ q^{(a-i)j}\gauss{a}{i}\gauss{b-1}{j} & \text{if} \ \chi\ne 1_G, \ \chi|_K=1_K, \\ q^{(a-i-1)j}\gauss{a-1}{i}\gauss{b}{j} & \text{if} \ \chi|_K\ne 1_K, \end{cases}
\end{equation*}
where $[\cdot,\cdot]$ denotes the usual inner product of characters.
\end{enumerate}
\end{lemma}

For $0\leqslant i\leqslant a$, $0\leqslant j\leqslant b$, and $\chi\in \widehat{G}$, let $M(\chi)$ and $M(\chi)_{i,j}$ be the homogeneous components of $\chi$ in $\mathbb{C}\tilde{P}_1$ and $\mathbb{C}\tilde{P}_{i,j,1}$, respectively.
Note that
\begin{equation}\label{dimension=multiplicity}
	\dim M(\chi)_{i,j}=[\chi,\psi_{i,j}].
\end{equation}
Hence it follows from Lemma \ref{characters}\,(ii) that
\begin{equation*}
	M(\chi)=\begin{cases} \bigoplus_{i=0}^a\bigoplus_{j=0}^b M(1_G)_{i,j} & \text{if} \ \chi=1_G,  \\ \bigoplus_{i=0}^a\bigoplus_{j=0}^{b-1} M(\chi)_{i,j} & \text{if} \ \chi\ne 1_G, \ \chi|_K=1_K, \\ \bigoplus_{i=0}^{a-1}\bigoplus_{j=0}^b M(\chi)_{i,j} & \text{if} \ \chi|_K\ne 1_K, \end{cases}
\end{equation*}
and that
\begin{equation}\label{dimensions of G-components}
	\dim M(\chi)=\begin{cases} s_q(a+b) & \text{if} \ \chi=1_G,  \\ s_q(a+b-1) & \text{if} \ \chi\ne 1_G, \end{cases}
\end{equation}
by virtue of \eqref{number of subspaces}.
Moreover,
\begin{equation*}
	\mathbb{C}\tilde{P}_1=\bigoplus_{\chi\in\widehat{G}}M(\chi).
\end{equation*}
Observe that $G$ acts trivially on $\tilde{P}_0$, and that the generators of $\tilde{\mathcal{H}}$ commute with the action of $G$, from which it follows that $\mathbb{C}\tilde{P}_0\bigoplus M(1_G)$ and the $M(\chi)$ $(\chi\ne 1_G)$ are $\tilde{\mathcal{H}}$-modules.

For $\chi\in\widehat{G}$, let
\begin{equation*}
	e_{\chi}=\frac{1}{q^{a+b}}\sum_{g\in G}\chi(g^{-1})g\in\mathbb{C}G.
\end{equation*}
Note that the $e_{\chi}$ are the (central) primitive idempotents of the group algebra $\mathbb{C}G$.
In particular, we have
\begin{equation}\label{apply central primitive idempotent}
	e_{\chi}\mathbb{C}\tilde{P}_{i,j,1}=M(\chi)_{i,j} \qquad (0\leqslant i\leqslant a,\, 0\leqslant j\leqslant b).
\end{equation}

\begin{lemma}[{\cite[Lemma 2.4]{Srinivasan2014EJC}}]\label{projection}
Let $y\in\tilde{P}_1$ and $\chi\in\widehat{G}$.
Then $e_{\chi}y\ne 0$ if and only if $\chi|_{G_y}=1_{G_y}$.
\end{lemma}

For the rest of this section, we will fix $u\in \tilde{P}_{0,0,1}$.
For $y\in\tilde{P}_0$, we will use $y\vee u$ to denote the subspace of $\mathbb{F}_q^{a+b+1}$ spanned by $y$ and $u$, in order to avoid confusion with the addition in $\mathbb{C}\tilde{P}$.
The following is essentially from \cite[Theorem 2.5\,(i)--(iii)]{Srinivasan2014EJC}.

\begin{lemma}
\label{basis}
Let $\chi\in\widehat{G}$.
Then the following hold:
\begin{enumerate}
\item Let $y,z\in\tilde{P}_1$. If $y\sim z$ then $e_{\chi}y$ is a non-zero scalar multiple of $e_{\chi}z$.
\item Let $y,z\in\tilde{P}_1$. If $y\subset z$ and $e_{\chi}z\ne 0$ then $e_{\chi}y\ne 0$.
\item For $0\leqslant i\leqslant a$ and $0\leqslant j\leqslant b$, the non-zero vectors among the vectors
\begin{equation*}
	e_{\chi}(y\vee u) \qquad (y\in\tilde{P}_{i,j,0})
\end{equation*}
form a basis of $M(\chi)_{i,j}$.
\end{enumerate}
\end{lemma}

\begin{proof}
(i), (ii): See \cite[Theorem 2.5\,(i), (iii)]{Srinivasan2014EJC}.

(iii): Note that the subspaces $y\vee u$ $(y\in\tilde{P}_{i,j,0})$ form a complete set of representatives of the equivalence classes in $\tilde{P}_{i,j,1}$.
From (i) above and \eqref{apply central primitive idempotent}, it follows that these vectors span $M(\chi)_{i,j}$.
Moreover, they have mutually disjoint supports, and hence are linearly independent.
\end{proof}

For any subspaces $x'$ and $H'$ of $\mathbb{F}_q^{a+b+1}$ with $x'\subset H'$, we let $P(x',H')$ denote the set of all subspaces of $H'$, to which we attach the matrices $\mathcal{L}_1,\mathcal{L}_2,\mathcal{R}_1,\mathcal{R}_2,\mathcal{E}_{i,j}^*\in\operatorname{Mat}_{P(x',H')}(\mathbb{C})$ with respect to the fixed subspace $x'$ as in the previous section.
(For these matrices, the underlying subspaces $x'$ and $H'$ will be clear from the context.)
We will also consider the corresponding algebra $\mathcal{H}(x',H')$ generated by these matrices.
The following extends \cite[Theorem 2.5\,(iv), (v)]{Srinivasan2014EJC}.

\begin{proposition}\label{how tilde-H is related to H}
The following hold:
\begin{enumerate}
\item The matrix $\tilde{R}_3$ gives a vector space isomorphism from $\mathbb{C}P(x,H)(=\mathbb{C}\tilde{P}_0)$ to $M(1_G)$, where we have, on $\mathbb{C}P(x,H)$,
\begin{gather*}
	\tilde{\mathcal{L}}_1\tilde{\mathcal{R}}_3=\tilde{\mathcal{R}}_3\mathcal{L}_1, \qquad \tilde{\mathcal{L}}_2\tilde{\mathcal{R}}_3=\tilde{\mathcal{R}}_3\mathcal{L}_2, \qquad \tilde{\mathcal{L}}_3\tilde{\mathcal{R}}_3\mathcal{E}_{i,j}^*=q^{a+b-i-j}\mathcal{E}_{i,j}^*, \\ \tilde{\mathcal{R}}_1\tilde{\mathcal{R}}_3=q\tilde{\mathcal{R}}_3\mathcal{R}_1, \qquad \tilde{\mathcal{R}}_2\tilde{\mathcal{R}}_3=q\tilde{\mathcal{R}}_3\mathcal{R}_2, \qquad  \tilde{\mathcal{R}}_3\tilde{\mathcal{R}}_3=0,\\
	\tilde{\mathcal{E}}_{i,j,0}^*\tilde{\mathcal{R}}_3=0, \qquad \tilde{\mathcal{E}}_{i,j,1}^*\tilde{\mathcal{R}}_3=\tilde{\mathcal{R}}_3\mathcal{E}_{i,j}^*
\end{gather*}
for $0\leqslant i\leqslant a$ and $0\leqslant j\leqslant b$.
\item Let $\chi\in\widehat{G}$ with $\chi\ne 1_G$ and $\chi|_K=1_K$.
Define the linear map $\Theta_{\chi}:\mathbb{C}\tilde{P}_0\rightarrow\mathbb{C}\tilde{P}_1$ by
\begin{equation*}
	\Theta_{\chi}y=e_{\chi}(y\vee u) \qquad (y\in\tilde{P}_0).
\end{equation*}
Then there is a unique $H_{\chi}\in\tilde{P}_{a,b-1,0}$ such that for every $y\in\tilde{P}_0$, we have $\Theta_{\chi}y\ne 0$ if and only if $y\subset H_{\chi}$.
Moreover, $\Theta_{\chi}$ gives a vector space isomorphism from $\mathbb{C}P(x,H_{\chi})$ to $M(\chi)$, where we have, on $\mathbb{C}P(x,H_{\chi})$,
\begin{gather*}
	\tilde{\mathcal{L}}_1\Theta_{\chi}=q\Theta_{\chi}\mathcal{L}_1, \qquad \tilde{\mathcal{L}}_2\Theta_{\chi}=q\Theta_{\chi}\mathcal{L}_2, \qquad \tilde{\mathcal{L}}_3\Theta_{\chi}=0, \\
	\tilde{\mathcal{R}}_1\Theta_{\chi}=\Theta_{\chi}\mathcal{R}_1, \qquad \tilde{\mathcal{R}}_2\Theta_{\chi}=\Theta_{\chi}\mathcal{R}_2, \qquad \tilde{\mathcal{R}}_3\Theta_{\chi}=0, \\
	\tilde{\mathcal{E}}_{i,j,0}^*\Theta_{\chi}=\tilde{\mathcal{E}}_{i,b,0}^*\Theta_{\chi}=\tilde{\mathcal{E}}_{i,b,1}^*\Theta_{\chi}=0, \qquad \tilde{\mathcal{E}}_{i,j,1}^*\Theta_{\chi}=\Theta_{\chi}\mathcal{E}_{i,j}^*
\end{gather*}
for $0\leqslant i\leqslant a$ and $0\leqslant j\leqslant b-1$.
\item Let $\chi\in\widehat{G}$ with $\chi|_K\ne 1_K$.
Let $\Theta_{\chi}:\mathbb{C}\tilde{P}_0\rightarrow\mathbb{C}\tilde{P}_1$ be as in (ii) above.
Then there is a unique $H_{\chi}\in\tilde{P}_{a-1,b,0}$ such that for every $y\in\tilde{P}_0$, we have $\Theta_{\chi}y\ne 0$ if and only if $y\subset H_{\chi}$.
Moreover, $\Theta_{\chi}$ gives a vector space isomorphism from $\mathbb{C}P(x\cap H_{\chi},H_{\chi})$ to $M(\chi)$, where we have, on $\mathbb{C}P(x\cap H_{\chi},H_{\chi})$,
\begin{gather*}
	\tilde{\mathcal{L}}_1\Theta_{\chi}=q\Theta_{\chi}\mathcal{L}_1, \qquad \tilde{\mathcal{L}}_2\Theta_{\chi}=q\Theta_{\chi}\mathcal{L}_2, \qquad \tilde{\mathcal{L}}_3\Theta_{\chi}=0, \\
	\tilde{\mathcal{R}}_1\Theta_{\chi}=\Theta_{\chi}\mathcal{R}_1, \qquad \tilde{\mathcal{R}}_2\Theta_{\chi}=\Theta_{\chi}\mathcal{R}_2, \qquad \tilde{\mathcal{R}}_3\Theta_{\chi}=0, \\
	\tilde{\mathcal{E}}_{i,j,0}^*\Theta_{\chi}=\tilde{\mathcal{E}}_{a,j,0}^*\Theta_{\chi}=\tilde{\mathcal{E}}_{a,j,1}^*\Theta_{\chi}=0, \qquad \tilde{\mathcal{E}}_{i,j,1}^*\Theta_{\chi}=\Theta_{\chi}\mathcal{E}_{i,j}^*
\end{gather*}
for $0\leqslant i\leqslant a-1$ and $0\leqslant j\leqslant b$.
\end{enumerate}
\end{proposition}

\begin{proof}
(i): Let $y\in\tilde{P}_{i,j,0}$.
Then there are $q^{a+b-i-j}$ elements $z\in\tilde{P}_{i,j,1}$ such that $y\subset z$, or equivalently, $H\cap z=y$.
Hence we have $\tilde{\mathcal{L}}_3\tilde{\mathcal{R}}_3y=q^{a+b-i-j}y$, from which it follows that $\tilde{\mathcal{L}}_3\tilde{\mathcal{R}}_3\mathcal{E}_{i,j}^*=q^{a+b-i-j}\mathcal{E}_{i,j}^*$.
Since $\mathbb{C}P(x,H)$ and $M(1_G)$ have the same dimension by \eqref{dimensions of G-components}, it follows that $\tilde{R}_3$ gives a vector space isomorphism from $\mathbb{C}P(x,H)$ to $M(1_G)$.
The other identities are easily verified.

(ii): By Lemma \ref{characters}\,(ii) and \eqref{dimension=multiplicity}, we have $	\dim M(\chi)_{a,b-1}=1$.
Hence it follows from Lemma \ref{basis}\,(iii) that there is a unique $H_{\chi}\in\tilde{P}_{a,b-1,0}$ such that $\Theta_{\chi}H_{\chi}\ne 0$.
Since $\mathbb{C}P(x,H_{\chi})$ and $M(\chi)$ have the same dimension by \eqref{dimensions of G-components}, it follows from Lemma \ref{basis}\,(ii) and (iii) that $\Theta_{\chi}$ gives a vector space isomorphism from $\mathbb{C}P(x,H_{\chi})$ to $M(\chi)$, and that, for every $y\in\tilde{P}_0$, we have $\Theta_{\chi}y\ne 0$ if and only if $y\subset H_{\chi}$.

Let $y\in\tilde{P}_{i,j,0}$ be such that $y\subset H_{\chi}$.
On the one hand, we have
\begin{equation*}
	\Theta_{\chi}\mathcal{R}_1y=\sum_{\substack{z\in\tilde{P}_{i+1,j,0} \\ y\subset z\subset H_{\chi}}} \!\!\Theta_{\chi}z.
\end{equation*}
On the other hand, since $\tilde{\mathcal{R}}_1$ and $e_{\chi}$ commute, we have
\begin{equation*}
	\tilde{\mathcal{R}}_1\Theta_{\chi}y=e_{\chi}\tilde{\mathcal{R}}_1(y\vee u)=e_{\chi}\!\sum_{\substack{w\in\tilde{P}_{i+1,j,1} \\ y\vee u\subset w}}\!\!\!w=e_{\chi}\!\sum_{\substack{z\in\tilde{P}_{i+1,j,0} \\ y\subset z}}\!\!\!(z\vee u)=\sum_{\substack{z\in\tilde{P}_{i+1,j,0} \\ y\subset z}}\!\!\Theta_{\chi}z.
\end{equation*}
However, in the last sum above, we have $\Theta_{\chi}z=0$ unless $z\subset H_{\chi}$.
Hence it follows that the above two vectors are equal.
This proves that $\tilde{\mathcal{R}}_1\Theta_{\chi}=\Theta_{\chi}\mathcal{R}_1$.
Likewise, we have
\begin{equation*}
	\Theta_{\chi}\mathcal{L}_1y=\sum_{\substack{z\in\tilde{P}_{i-1,j,0} \\ z\subset y}} \!\!\Theta_{\chi}z,
\end{equation*}
and also
\begin{equation*}
	\tilde{\mathcal{L}}_1\Theta_{\chi}y=e_{\chi}\tilde{\mathcal{L}}_1(y\vee u)=e_{\chi}\!\sum_{\substack{w\in\tilde{P}_{i-1,j,1} \\ w\subset y\vee u}}\!\!\!w=\sum_{\substack{z\in\tilde{P}_{i-1,j,0} \\ z\subset y}}\sum_{\substack{w\in\tilde{P}_{i-1,j,1} \\ z\subset w\subset y\vee u}}\!\!e_{\chi}w.
\end{equation*}
In the last sum above, for each $z$, there are exactly $q$ choices for $w$.
Pick any such $w$.
Then since $w\sim z\vee u$, there exists $g\in G$ such that $w=g(z\vee u)$.
However, since both $w$ and $z\vee u$ are subspaces of $y\vee u$, this $g$ must fix $y\vee u$, i.e., $g\in G_{y\vee u}$.
On the other hand, recall that $\Theta_{\chi}y\ne 0$ since $y\subset H_{\chi}$.
Hence it follows from Lemma \ref{projection} that $\chi|_{G_{y\vee u}}=1_{G_{y\vee u}}$.
In particular, we have $\chi(g)=1$, and hence
\begin{equation*}
	e_{\chi}w=e_{\chi}g(z\vee u)=e_{\chi}(z\vee u)=\Theta_{\chi}z.
\end{equation*}
Combining these comments, we have $\tilde{\mathcal{L}}_1\Theta_{\chi}y=q\Theta_{\chi}\mathcal{L}_1y$, and consequently, $\tilde{\mathcal{L}}_1\Theta_{\chi}=q\Theta_{\chi}\mathcal{L}_1$.
The identities involving $\tilde{\mathcal{L}}_2$ and $\tilde{\mathcal{R}}_2$ are proved similarly.
Let $y$ be as above.
Then, since $\tilde{\mathcal{L}}_3$ and $e_{\chi}$ commute, we have
\begin{equation*}
	\tilde{\mathcal{L}}_3\Theta_{\chi}y=e_{\chi}\tilde{\mathcal{L}}_3(y\vee u)=e_{\chi}y=\frac{1}{q^{a+b}}\sum_{g\in G}\chi(g^{-1})y=0,
\end{equation*}
which proves that $\tilde{\mathcal{L}}_3\Theta_{\chi}=0$.
The other identities are trivial.

(iii): Similar to the proof of (ii) above.
\end{proof}

Let $\tilde{\mathcal{W}}$ be an irreducible $\tilde{\mathcal{H}}$-module.
Let
\begin{align*}
	\nu &= \min\{i:(\tilde{\mathcal{E}}_{i,0,0}^*+\dots+\tilde{\mathcal{E}}_{i,b,1}^*)\tilde{\mathcal{W}}\ne 0\}, \\
	\nu' &= \max\{i:(\tilde{\mathcal{E}}_{i,0,0}^*+\dots+\tilde{\mathcal{E}}_{i,b,1}^*)\tilde{\mathcal{W}}\ne 0\}, \\
	\mu &= \min\{j:(\tilde{\mathcal{E}}_{0,j,0}^*+\dots+\tilde{\mathcal{E}}_{a,j,1}^*)\tilde{\mathcal{W}}\ne 0\}, \\
	\mu' &= \max\{j:(\tilde{\mathcal{E}}_{0,j,0}^*+\dots+\tilde{\mathcal{E}}_{a,j,1}^*)\tilde{\mathcal{W}}\ne 0\}, \\
	\tau &= \min\{k:(\tilde{\mathcal{E}}_{0,0,k}^*+\dots+\tilde{\mathcal{E}}_{a,b,k}^*)\tilde{\mathcal{W}}\ne 0\}, \\
	\tau' &= \max\{k:(\tilde{\mathcal{E}}_{0,0,k}^*+\dots+\tilde{\mathcal{E}}_{a,b,k}^*)\tilde{\mathcal{W}}\ne 0\}.
\end{align*}
We call $(\nu,\mu,\tau)$ (resp.~$(\nu',\mu',\tau')$) the \emph{lower endpoint} (resp.~\emph{upper endpoint}) of $\tilde{\mathcal{W}}$.
We let
\begin{equation*}
	\rho=a-\nu-\nu',
\end{equation*}
and call $\rho$ the \emph{index} of $\tilde{\mathcal{W}}$.
We now state the main result of this section.

\begin{theorem}\label{irreducible tilde-H-modules}
There exists an irreducible $\tilde{\mathcal{H}}$-module with lower endpoint $(\nu,\mu,\tau)$, upper endpoint $(\nu',\mu',\tau')$, and index $\rho$ if and only if
\begin{gather*}
	0\leqslant \nu\leqslant a, \qquad 0\leqslant \mu\leqslant b, \qquad 0\leqslant \tau \leqslant 1, \qquad \tau'=1, \\
	\qquad \max\{0,2\mu-b+\tau\} \leqslant \rho =\mu+\mu'+\tau-b \leqslant \min\{a-2\nu,\mu+\tau\}.
\end{gather*}
Let $\tilde{\mathcal{W}}$ be an irreducible $\tilde{\mathcal{H}}$-module with lower endpoint $(\nu,\mu,\tau)$ and index $\rho$.
Then the isomorphism class of $\tilde{\mathcal{W}}$ is determined by $\nu,\mu,\tau$, and $\rho$, and the following hold:
\begin{enumerate}
\item If $\tau=0$, then $\tilde{\mathcal{W}}$ has a basis $w_{i,j,k}$ $(\nu\leqslant i\leqslant a-\nu-\rho, \ \mu\leqslant j\leqslant b-\mu+\rho, \ 0\leqslant k\leqslant 1)$ such that $w_{i,j,k}\in\tilde{\mathcal{E}}_{i,j,k}^*\tilde{\mathcal{W}}$, and
\begin{align*}
	\tilde{\mathcal{L}}_1w_{i,j,k}&=q^{\nu+j}\gauss{a-\nu-\rho-i+1}{1}\gauss{i-\nu}{1}w_{i-1,j,k}, & \tilde{\mathcal{R}}_1w_{i,j,k}&=q^kw_{i+1,j,k}, \\
	\tilde{\mathcal{L}}_2w_{i,j,k}&=q^{a-\nu+\mu-\rho}\gauss{b-\mu+\rho-j+1}{1}\gauss{j-\mu}{1}w_{i,j-1,k}, & \tilde{\mathcal{R}}_2w_{i,j,k}&=q^{\nu-i+k}w_{i,j+1,k}, \\
	\tilde{\mathcal{L}}_3w_{i,j,k}&=q^{a+b-i-j}w_{i,j,k-1}, & \tilde{\mathcal{R}}_3w_{i,j,k}&=w_{i,j,k+1}
\end{align*}
for all $i,j,k$, where we set $w_{i,j,k}:=0$ if $(i,j,k)$ is outside the above range.
Moreover, in this case, the multiplicity $m_{\nu,\mu,0,\rho}$ of $\tilde{\mathcal{W}}$ in $\mathbb{C}\tilde{P}$ is given by
\begin{align*}
	m_{\nu,\mu,0,\rho}&=\frac{ (-1)^{\rho} (q)_a (q)_b (1-q^{a-2\nu-\rho+1}) (1-q^{b-2\mu+\rho+1}) q^{\nu+\mu-\rho+\binom{\rho}{2}} }{ (q)_{a-\nu-\rho+1} (q)_{b-\mu+1} (q)_{\nu} (q)_{\mu-\rho} (q)_{\rho}}.
\end{align*}
\item If $\tau=1$, then $\tilde{\mathcal{W}}$ has a basis $w_{i,j,1}$ $(\nu\leqslant i\leqslant a-\nu-\rho, \ \mu\leqslant j\leqslant b-\mu+\rho-1)$ such that $w_{i,j,1}\in\tilde{\mathcal{E}}_{i,j,1}^*\tilde{\mathcal{W}}$, and
\begin{align*}
	\tilde{\mathcal{L}}_1w_{i,j,1}&=q^{\nu+j+1}\gauss{a-\nu-\rho-i+1}{1}\gauss{i-\nu}{1}w_{i-1,j,1}, & \tilde{\mathcal{R}}_1w_{i,j,1}&=w_{i+1,j,1}, \\
	\tilde{\mathcal{L}}_2w_{i,j,1}&=q^{a-\nu+\mu-\rho+1}\gauss{b-\mu+\rho-j}{1}\gauss{j-\mu}{1}w_{i,j-1,1}, & \tilde{\mathcal{R}}_2w_{i,j,1}&=q^{\nu-i}w_{i,j+1,1}, \\
	\tilde{\mathcal{L}}_3w_{i,j,1}&=0, & \tilde{\mathcal{R}}_3w_{i,j,1}&=0
\end{align*}
for all $i,j$, where we set $w_{i,j,1}:=0$ if $(i,j)$ is outside the above range.
Moreover, in this case, the multiplicity $m_{\nu,\mu,1,\rho}$ of $\tilde{\mathcal{W}}$ in $\mathbb{C}\tilde{P}$ is given by
\begin{align*}
	m_{\nu,\mu,1,\rho}&=\frac{ (-1)^{\rho+1} (q)_a (q)_b (1-q^{a-2\nu-\rho+1}) (1-q^{b-2\mu+\rho}) q^{\nu+\mu-\rho+\binom{\rho}{2}} }{ (q)_{a-\nu-\rho+1} (q)_{b-\mu+1} (q)_{\nu} (q)_{\mu-\rho+1} (q)_{\rho} } \\
	& \qquad \times (q^{b+2}-q^{b-\mu+1}-q^{\mu-\rho+1}+1).
\end{align*}
\end{enumerate}
\end{theorem}

\begin{proof}
First, recall that $\mathbb{C}\tilde{P}_0\bigoplus M(1_G)$ is an $\tilde{\mathcal{H}}$-module.
Recall also the algebra $\mathcal{H}(x,H)$ acting on $\mathbb{C}P(x,H)=\mathbb{C}\tilde{P}_0$.
Let $\mathcal{W}$ be an irreducible $\mathcal{H}(x,H)$-module in $\mathbb{C}P(x,H)$ with lower endpoint $(\nu,\mu)$ and index $\rho$, and let the basis vectors $w_{i,j}$ be as in Theorem \ref{irreducible H-modules}.
For $\nu\leqslant i\leqslant a-\nu-\rho$ and $\mu\leqslant j\leqslant b-\mu+\rho$, let $w_{i,j,0}=w_{i,j}$ and  $w_{i,j,1}=\tilde{\mathcal{R}}_3w_{i,j}$.
Then from Proposition \ref{how tilde-H is related to H}\,(i) it follows that the $w_{i,j,k}$ are non-zero and are linearly independent, and that they form a basis of an $\tilde{\mathcal{H}}$-module, which we denote by $\tilde{\mathcal{W}}$.
It is also immediate to see that the actions of the generators on the $w_{i,j,k}$ are as in (i).
We now claim that $\tilde{\mathcal{W}}$ is irreducible.
Since $\tilde{\mathcal{H}}$ is semisimple, $\tilde{\mathcal{W}}$ is a direct sum of irreducible $\tilde{\mathcal{H}}$-submodules.
Since $\tilde{\mathcal{E}}^*_{\nu,\mu,0}\tilde{\mathcal{W}}=\operatorname{span}\{w_{\nu,\mu,0}\}\ne 0$, there is an irreducible $\tilde{\mathcal{H}}$-submodule $\tilde{\mathcal{U}}$ of $\tilde{\mathcal{W}}$ such that $\tilde{\mathcal{E}}^*_{\nu,\mu,0}\,\tilde{\mathcal{U}}\ne 0$.
Then we have $w_{\nu,\mu,0}\in\tilde{\mathcal{U}}$, and hence $\tilde{\mathcal{W}}=\tilde{\mathcal{H}}w_{\nu,\mu,0}\subset \tilde{\mathcal{U}}$, i.e., $\tilde{\mathcal{W}}=\tilde{\mathcal{U}}$.
It follows that $\tilde{\mathcal{W}}$ is irreducible.
We note that $\tilde{\mathcal{W}}$ has lower endpoint $(\nu,\mu,0)$ and index $\rho$.

Second, let $\chi\in\widehat{G}$ with $\chi\ne 1_G$ and $\chi|_K=1_K$, and recall the $\tilde{\mathcal{H}}$-module $M(\chi)$.
In this case, we consider the algebra $\mathcal{H}(x,H_{\chi})$, where $H_{\chi}\in\tilde{P}_{a,b-1,0}$ is from Proposition \ref{how tilde-H is related to H}\,(ii).
We argue as above by letting $\tilde{\mathcal{W}}$ be the linear span of the vectors $w_{i,j,1}=\Theta_{\chi}w_{i,j}$ $(\nu\leqslant i\leqslant a-\nu-\rho,\ \mu\leqslant j\leqslant b-\mu+\rho-1)$.
The actions of the generators on the $w_{i,j,1}$ are as in (ii).
The irreducible $\tilde{\mathcal{H}}$-module $\tilde{\mathcal{W}}$ has lower endpoint $(\nu,\mu,1)$ and index $\rho$.

Third, let $\chi\in\widehat{G}$ with $\chi|_K\ne 1_K$.
In this case, we consider the algebra $\mathcal{H}(x\cap H_{\chi},H_{\chi})$, where $H_{\chi}\in\tilde{P}_{a-1,b,0}$ is from Proposition \ref{how tilde-H is related to H}\,(iii).
We again argue as above, but we start from an irreducible $\mathcal{H}(x\cap H_{\chi},H_{\chi})$-module $\mathcal{W}$ with lower endpoint $(\nu,\mu)$ and index $\rho-1$.
Then the actions of the generators on the $w_{i,j,1}$ are again as in (ii), and the irreducible $\tilde{\mathcal{H}}$-module $\tilde{\mathcal{W}}$ has lower endpoint $(\nu,\mu,1)$ and index $\rho$.

The formulas for the multiplicities follow from Theorem \ref{irreducible H-modules}.
Note that we obtain isomorphic irreducible $\tilde{\mathcal{H}}$-modules from the second and the third cases above, so that the multiplicity in (ii) is computed by adding two terms, each multiplied by the number of respective characters, and then simplifying.
The other statements are also verified using Theorem \ref{irreducible H-modules}.
\end{proof}

\section{The Terwilliger algebra \texorpdfstring{$T=T(x)$}{T=T(x)}}
\label{sec: irreducible T-modules}

We now turn to the discussions on the Terwilliger algebra $T=T(x)$ of the twisted Grassmann graph $\tilde{J}_q(2D+1,D)$, where we choose the base vertex $x$ from $X''$.
We mentioned in Introduction that $\tilde{J}_q(2D+1,D)$ has the same intersection array as the Grassmann graph $J_q(2D+1,D)$.
In particular, it is an example of a $Q$-\emph{polynomial} distance-regular graph with diameter $D$.
(For the background information on distance-regular graphs, we refer to \cite{BI1984B,BCN1989B,DKT2016EJC,Godsil1993B}.)
The eigenvalues of $J_q(2D+1,D)$, and hence of $\tilde{J}_q(2D+1,D)$, are given in \cite[Theorem 9.3.3]{BCN1989B} as follows:
\begin{equation*}
	\theta_i=q\gauss{D}{1}\gauss{D+1}{1}-\gauss{i}{1}\gauss{2D-i+2}{1} \qquad (0\leqslant i\leqslant D).
\end{equation*}
Recall the diagonal matrices $E_i^*=E_i^*(x)\in\operatorname{Mat}_X(\mathbb{C})$ $(0\leqslant i\leqslant D)$.
We note that two vertices $y$ and $z$ are at distance $i$ if and only if (cf.~\eqref{adjacency})
\begin{equation}\label{distance i}
	\dim y+\dim z-2\dim y\cap z=2i.
\end{equation}
For $0\leqslant i\leqslant D$, let $E_i\in\operatorname{Mat}_X(\mathbb{C})$ be the orthogonal projection onto the eigenspace of the adjacency matrix $A$ with eigenvalue $\theta_i$.

Let $W$ be an irreducible $T$-module.
Define the \emph{support} and the \emph{dual support} of $W$ by
\begin{equation*}
	W_s=\{i:E_i^*W\ne 0\}, \qquad W_s^*=\{i:E_iW\ne 0\},
\end{equation*}
respectively.
We then define the \emph{endpoint}, \emph{dual endpoint}, \emph{diameter}, and the \emph{dual diameter} of $W$ by
\begin{equation*}
	\epsilon=\min W_s, \qquad \epsilon^*=\min W_s^*, \qquad d=|W_s|-1, \qquad d^*=|W_s^*|-1,
\end{equation*}
respectively.
It is known that $d=d^*$, and that
\begin{equation*}
	W_s=\{\epsilon,\epsilon+1,\dots,\epsilon+d\}, \qquad W_s^*=\{\epsilon^*,\epsilon^*+1,\dots,\epsilon^*+d\}.
\end{equation*}
See \cite[Corollary 3.3]{Pascasio2002EJC} and \cite[Lemmas 3.9, 3.12]{Terwilliger1992JAC}.
Moreover, we have
\begin{gather}
	\epsilon\geqslant 0, \qquad \epsilon^*\geqslant 0, \qquad d\geqslant 0, \label{trivial condition 1} \\
	\epsilon+d\leqslant D, \qquad \epsilon^*+d\leqslant D, \label{trivial condition 2} \\
	2\epsilon+d\geqslant D, \qquad 2\epsilon^*+d\geqslant D. \label{trivial condition 3}
\end{gather}
Here, \eqref{trivial condition 1} and \eqref{trivial condition 2} are clear, and \eqref{trivial condition 3} is given in \cite[Lemmas 5.1, 7.1]{Caughman1999DM}.
In describing the irreducible $T$-modules, we thus consider the following set:
\begin{equation}\label{Omega}
	\Omega=\Omega_D=\{(\epsilon,\epsilon^*,d)\in\mathbb{Z}^3 : \text{\eqref{trivial condition 1}, \eqref{trivial condition 2}, \eqref{trivial condition 3} hold}\}.
\end{equation}
Observe that the first two inequalities in \eqref{trivial condition 1} are consequences of \eqref{trivial condition 2} and \eqref{trivial condition 3}, so that we may replace \eqref{trivial condition 1} in the definition of $\Omega$ by the following:
\begin{equation*}
	d\geqslant 0. \tag*{(\ref{trivial condition 1})'}
\end{equation*}

It will also be important to consider the following $(d+1)\times (d+1)$ matrix having four free parameters (besides $q$ and $d$):
\begin{equation}\label{standard form}
	\mathsf{A}(q,d;h,r,s,\lambda_0)=\begin{pmatrix} \mathsf{a}_0 & \mathsf{b}_0 &&&& \bm{0} \\ \mathsf{c}_1 & \mathsf{a}_1 & \mathsf{b}_1 &&& \\ & \mathsf{c}_2 & \mathsf{a}_2 & \cdot && \\ && \cdot & \cdot & \cdot & \\ &&& \cdot & \cdot & \mathsf{b}_{d-1} \\ \bm{0} &&&& \mathsf{c}_d & \mathsf{a}_d \end{pmatrix},
\end{equation}
where
\begin{gather*}
	\mathsf{a}_i+\mathsf{b}_i+\mathsf{c}_i=\lambda_0 \qquad (0\leqslant i\leqslant d), \\
	\mathsf{b}_i=h(1-q^{i-d})(1-rq^{i+1}), \qquad \mathsf{c}_i=hsq(1-q^i)(1-rq^{i-d-1}/s) \qquad (0\leqslant i\leqslant d),
\end{gather*}
and where we assume that $h,r,s\ne 0$, that $rq^i,sq^i/r\ne 1$ $(1\leqslant i\leqslant d)$, and that $sq^i\ne 1$ $(2\leqslant i\leqslant 2d)$.
We note that $\mathsf{b}_{i-1}\mathsf{c}_i\ne 0$ $(1\leqslant i\leqslant d)$.
This matrix is the standardized form of one of the operators of a \emph{Leonard system} of dual $q$-Hahn type; cf.~\cite[Example 5.5]{Terwilliger2005DCC}.
With the notation of \cite[Section 2]{Terwilliger1992JAC}, this also corresponds to Case (I) with $r_1=s^*=0$.
See also \cite[Theorem 17.7]{Terwilliger2004LAA}.
The eigenvalues of the matrix \eqref{standard form} are given by
\begin{equation}\label{eigenvalues}
	\lambda_i=\lambda_0+h(1-q^i)(1-sq^{i+1})q^{-i} \qquad (0\leqslant i\leqslant d).
\end{equation}
In particular, if we fix $h,s$, and $\lambda_0$, then the matrices \eqref{standard form} with distinct values of $r$ are all similar.
On the other hand, we note that

\begin{lemma}\label{uniqueness of r}
Suppose that $d>0$.
Then there exists a $(d+1)\times (d+1)$ invertible diagonal matrix $\mathsf{S}$ such that
\begin{equation*}
	\mathsf{S}^{-1}\mathsf{A}(q,d;h,r,s,\lambda_0)\mathsf{S}=\mathsf{A}(q,d;h,r',s,\lambda_0)
\end{equation*}
if and only if $r=r'$, in which case $\mathsf{S}$ is a non-zero scalar matrix.
\end{lemma}

\begin{proof}
Observe that the $\mathsf{a}_i$ are linear in $r$ provided that $d>0$.
Since conjugation by $\mathsf{S}$ does not change the diagonal entries, the result follows.
\end{proof}

\noindent
If $d=0$ then the matrix \eqref{standard form} is in fact independent of $h,r$, and $s$, but we will still include these for convenience of the descriptions.

For the rest of this paper, we retain the notation of Section \ref{sec: extension of H} with
\begin{equation*}
	a=D-1, \qquad b=D+1,
\end{equation*}
where we take the base vertex $x\in X''$ as the fixed subspace $x\in\tilde{P}$ in Section \ref{sec: extension of H}, and similarly for the hyperplane $H$.
Then we have
\begin{equation*}
	X'=\bigsqcup_{\ell=1}^D \tilde{P}_{D-\ell,\ell,1}, \qquad X''=\bigsqcup_{\ell=0}^{D-1} \tilde{P}_{D-\ell-1,\ell,0}.
\end{equation*}
Set
\begin{equation*}
	\tilde{\mathcal{E}}^{*\prime}=\sum_{\ell=1}^D\tilde{\mathcal{E}}^*_{D-\ell,\ell,1}, \qquad \tilde{\mathcal{E}}^{*\prime\prime}=\sum_{\ell=0}^{D-1}\tilde{\mathcal{E}}^*_{D-\ell-1,\ell,0}, \qquad \tilde{\mathcal{E}}^*=\tilde{\mathcal{E}}^{*\prime}+\tilde{\mathcal{E}}^{*\prime\prime}.
\end{equation*}
Note that $\tilde{\mathcal{E}}^*$ is the orthogonal projection onto $\mathbb{C}X\subset\mathbb{C}\tilde{P}$.
We also identify $\operatorname{Mat}_X(\mathbb{C})$ with $\tilde{\mathcal{E}}^*\operatorname{Mat}_{\tilde{P}}(\mathbb{C})\tilde{\mathcal{E}}^*$ in the obvious manner.
With this notation, the adjacency matrix $A\in\operatorname{Mat}_X(\mathbb{C})$ of $\tilde{J}_q(2D+1,D)$ is written as
\begin{equation*}
	A=\tilde{\mathcal{E}}^{*\prime}A\tilde{\mathcal{E}}^{*\prime}+\tilde{\mathcal{E}}^{*\prime}A\tilde{\mathcal{E}}^{*\prime\prime}+\tilde{\mathcal{E}}^{*\prime\prime}A\tilde{\mathcal{E}}^{*\prime}+\tilde{\mathcal{E}}^{*\prime\prime}A\tilde{\mathcal{E}}^{*\prime\prime},
\end{equation*}
and direct computations show that
\begin{align}
	\tilde{\mathcal{E}}^{*\prime}A\tilde{\mathcal{E}}^{*\prime}&=\tilde{\mathcal{E}}^{*\prime}(\tilde{\mathcal{R}}_1+\tilde{\mathcal{R}}_2+\tilde{\mathcal{R}}_3)(\tilde{\mathcal{L}}_1+\tilde{\mathcal{L}}_2+\tilde{\mathcal{L}}_3)\tilde{\mathcal{E}}^{*\prime}-\gauss{D+1}{1}\tilde{\mathcal{E}}^{*\prime}, \label{A11} \\
	\tilde{\mathcal{E}}^{*\prime}A\tilde{\mathcal{E}}^{*\prime\prime}&=\tilde{\mathcal{E}}^{*\prime}\tilde{\mathcal{R}}_3(\tilde{\mathcal{R}}_1+\tilde{\mathcal{R}}_2)\tilde{\mathcal{E}}^{*\prime\prime}, \\
	\tilde{\mathcal{E}}^{*\prime\prime}A\tilde{\mathcal{E}}^{*\prime}&=\tilde{\mathcal{E}}^{*\prime\prime}(\tilde{\mathcal{L}}_1+\tilde{\mathcal{L}}_2)\tilde{\mathcal{L}}_3\tilde{\mathcal{E}}^{*\prime}, \\
	\tilde{\mathcal{E}}^{*\prime\prime}A\tilde{\mathcal{E}}^{*\prime\prime}&=\tilde{\mathcal{E}}^{*\prime\prime}(\tilde{\mathcal{R}}_1+\tilde{\mathcal{R}}_2)(\tilde{\mathcal{L}}_1+\tilde{\mathcal{L}}_2)\tilde{\mathcal{E}}^{*\prime\prime}-\gauss{D-1}{1}\tilde{\mathcal{E}}^{*\prime\prime}. \label{A22}
\end{align}
Moreover, we have (cf.~\eqref{distance i})
\begin{equation}\label{Ei*}
	E_i^*=\tilde{\mathcal{E}}^*_{D-i,i,1}+\tilde{\mathcal{E}}^*_{D-i-1,i,0} \qquad (0\leqslant i\leqslant D),
\end{equation}
where $\tilde{\mathcal{E}}^*_{D,0,1}=\tilde{\mathcal{E}}^*_{-1,D,0}:=0$.
It follows that $T$ is a subalgebra of $\tilde{\mathcal{E}}^*\tilde{\mathcal{H}}\tilde{\mathcal{E}}^*$.

We will now find all the irreducible $T$-modules in $\mathbb{C}X$.
First, let $\tilde{\mathcal{W}}$ be an irreducible $\tilde{\mathcal{H}}$-module in $\mathbb{C}\tilde{P}$ with lower endpoint $(\nu,\mu,0)$ and index $\rho$, where we recall from Theorem \ref{irreducible tilde-H-modules} that $\nu,\mu$, and $\rho$ satisfy
\begin{gather}
	0\leqslant \nu\leqslant D-1, \qquad 0\leqslant \mu\leqslant D+1, \label{range1} \\
	\max\{0,2\mu-D-1\}\leqslant \rho\leqslant \min\{D-2\nu-1,\mu\}. \label{range2}
\end{gather}
In particular, $\nu+\mu\leqslant D$.
Let the basis vectors $w_{i,j,k}$ of $\tilde{\mathcal{W}}$ be as in Theorem \ref{irreducible tilde-H-modules}\,(i).
Note that
\begin{equation*}
	\tilde{\mathcal{E}}^*\tilde{\mathcal{W}}=\tilde{\mathcal{E}}^{*\prime}\tilde{\mathcal{W}}\bigoplus\tilde{\mathcal{E}}^{*\prime\prime}\tilde{\mathcal{W}},
\end{equation*}
and that
\begin{align*}
	\tilde{\mathcal{E}}^{*\prime}\tilde{\mathcal{W}} &= \operatorname{span}\bigl\{ w_{D-i,i,1}:\max\{\nu+\rho+1,\mu\}\leqslant i\leqslant \min\{D-\nu,D-\mu+\rho+1\}\bigr\}, \\
	\tilde{\mathcal{E}}^{*\prime\prime}\tilde{\mathcal{W}} &= \operatorname{span}\bigl\{ w_{D-i-1,i,0}:\max\{\nu+\rho,\mu\}\leqslant i\leqslant\min\{D-\nu-1,D-\mu+\rho+1\}\bigr\},
\end{align*}
where we always have $\tilde{\mathcal{E}}^{*\prime}\tilde{\mathcal{W}}\ne 0$, whereas  $\tilde{\mathcal{E}}^{*\prime\prime}\tilde{\mathcal{W}}\ne 0$ precisely when $\nu+\mu<D$.

Let
\begin{equation*}
	w_i=w_{D-i,i,1}+q^{\nu+i}\gauss{D-\nu-i}{1}w_{D-i-1,i,0},
\end{equation*}
for $\max\{\nu+\rho,\mu\}\leqslant i\leqslant \min\{D-\nu,D-\mu+\rho+1\}$, and let
\begin{equation*}
	\overline{w}_i=w_{D-i,i,1}-q^D\gauss{i-\nu-\rho}{1}w_{D-i-1,i,0},
\end{equation*}
for $\max\{\nu+\rho+1,\mu\}\leqslant i\leqslant \min\{D-\nu-1,D-\mu+\rho+1\}$,
where we set $w_{i,j,k}:=0$ whenever $(i,j,k)$ is outside the parameter range.
Observe by \eqref{Ei*} that $w_i,\overline{w}_i\in E_i^*\tilde{\mathcal{W}}$.
We define the subspaces $W_1$ and $W_2$ of $\tilde{\mathcal{E}}^*\tilde{\mathcal{W}}$ by
\begin{align*}
	W_1 &= \operatorname{span}\bigl\{ w_i:\max\{\nu+\rho,\mu\}\leqslant i\leqslant \min\{D-\nu,D-\mu+\rho+1\}\bigr\}, \\
	W_2 &= \operatorname{span}\bigl\{ \overline{w}_i:\max\{\nu+\rho+1,\mu\}\leqslant i\leqslant \min\{D-\nu-1,D-\mu+\rho+1\}\bigr\},
\end{align*}
where we always have $W_1\ne 0$, whereas $W_2\ne 0$ precisely when $\rho<D-2\nu-1$.
Then we have
\begin{equation*}
	\tilde{\mathcal{E}}^*\tilde{\mathcal{W}}=W_1\bigoplus W_2.
\end{equation*}
Moreover, it follows from Theorem \ref{irreducible tilde-H-modules}\,(i) and \eqref{A11}--\eqref{A22} that
\begin{align}
	Aw_i =& \left\{ q^{\nu+i}\gauss{D-\nu-i}{1}\gauss{i-\nu-\rho+1}{1} -\gauss{D}{1} \right. \notag \\
	& \qquad\qquad \left.+\,q^{\mu-\rho+i}\gauss{D-\mu+\rho-i+2}{1}\gauss{i-\mu}{1} \right\} w_i \notag \\
	& +q^{2\nu-D+2i+1}\gauss{D-\nu-i}{1}\gauss{i-\nu-\rho+1}{1} w_{i+1} \notag \\
	& +q^{D-\nu+\mu-\rho}\gauss{D-\mu+\rho-i+2}{1}\gauss{i-\mu}{1}w_{i-1}, \label{Awi}
\end{align}
and
\begin{align*}
	A\overline{w}_i =& \left\{ q^{\nu+i+1}\gauss{D-\nu-i-1}{1}\gauss{i-\nu-\rho}{1} -\gauss{D}{1} \right. \\
	& \qquad\qquad \left. +\,q^{\mu-\rho+i}\gauss{D-\mu+\rho-i+2}{1}\gauss{i-\mu}{1} \right\} \overline{w}_i \\
	& +q^{2\nu-D+2i+2}\gauss{D-\nu-i-1}{1}\gauss{i-\nu-\rho}{1} \overline{w}_{i+1} \\
	& +q^{D-\nu+\mu-\rho}\gauss{D-\mu+\rho-i+2}{1}\gauss{i-\mu}{1} \overline{w}_{i-1}
\end{align*}
for all $i$, where we understand that $w_i=\overline{w}_i=0$ whenever they are undefined.
It follows that $W_1$ and $W_2$ are $T$-modules.

We now claim that $W_1$ and $W_2$ are thin irreducible $T$-modules.
(For $W_2$, the claim holds under the additional assumption that $\rho<D-2\nu-1$; otherwise we have $W_2=0$.)
Let $U$ be a non-zero $T$-submodule of  $W_1$.
Since $U$ is closed under the $E_i^*$, and since $E_i^*W_1=\operatorname{span}\{w_i\}$, it follows that $U$ is spanned by some of the $w_i$.
Suppose that $w_i\in U$.
Observe that the coefficients of $w_{i\pm 1}$ in $Aw_i$ are non-zero whenever $w_{i\pm 1}$ are defined.
In other words, $E_{i\pm 1}^*Aw_i$ are non-zero scalar multiples of $w_{i\pm 1}$, and hence $w_{i\pm 1}\in U$.
By repeating this argument, it follows that $U$ contains all the basis vectors of $W_1$, i.e., $U=W_1$.
Hence $W_1$ is an irreducible $T$-module, and it is clear that $W_1$ is thin.
The same proof works for $W_2$ as well.

The endpoint $\epsilon$ and the diameter $d$ of $W_1$ are given by
\begin{equation*}
	(\epsilon,d)=\begin{cases} (\nu+\rho,D-2\nu-\rho) & \text{if} \ \nu+\rho\geqslant \mu, \\ (\mu,D-2\mu+\rho+1) & \text{if} \ \nu+\rho<\mu. \end{cases}
\end{equation*}
Consider the matrix \eqref{standard form} with parameters
\begin{equation*}
	h=\frac{q^{2D+2-\nu-\mu}}{(1-q)^2}, \qquad s=q^{2\nu+2\mu-2D-3}, \qquad \lambda_0=\theta_{\nu+\mu},
\end{equation*}
and
\begin{equation*}
	r=\begin{cases} q^{\nu+\mu-D-2} & \text{if} \ \nu+\rho\geqslant \mu, \\ q^{\nu+\mu-D-1} & \text{if} \ \nu+\rho<\mu. \end{cases}
\end{equation*}
Let $\gamma_{i+1}$ denote the coefficient of $w_{i+1}$ in $Aw_i$; cf.~\eqref{Awi}.
We mentioned above that $\gamma_{i+1}\ne 0$ for $\epsilon\leqslant i<\epsilon+d$.
Define the new basis $v_i$ $(0\leqslant i\leqslant d)$ of $W_1$ by
\begin{equation*}
	v_i=\left(\prod_{\ell=0}^{i-1}\frac{\gamma_{\epsilon+\ell+1}}{\mathsf{c}_{\ell+1}}\right) w_{\epsilon+i} \qquad (0\leqslant i\leqslant d).
\end{equation*}
Then we can routinely verify that $\mathsf{A}(q,d;h,r,s,\lambda_0)$ with the above parameters gives the matrix representing $A|_{W_1}$ with respect to the $v_i$, and that (cf.~\eqref{eigenvalues})
\begin{equation*}
	\lambda_i=\theta_{\nu+\mu+i} \qquad (0\leqslant i\leqslant d).
\end{equation*}
It follows that the dual endpoint $\epsilon^*$ of $W_1$ is given by
\begin{equation*}
	\epsilon^*=\nu+\mu.
\end{equation*}

\begin{theorem}\label{irreducible T-modules of type 1}
Let $V_1$ be the sum of all the $W_1$ obtained as above, where the $\tilde{\mathcal{W}}$ are over the irreducible $\tilde{\mathcal{H}}$-modules in $\mathbb{C}\tilde{P}$ with $\tau=0$.
Then we have
\begin{equation*}
	V_1=V_{1,0}\bigoplus V_{1,1},
\end{equation*}
where $V_{1,0}$ and $V_{1,1}$ are $T$-submodules of $V_1$ such that the following hold:
\begin{enumerate}
\item For the irreducible $T$-submodules $W$ in $V_{1,0}$, the endpoint $\epsilon$, dual endpoint $\epsilon^*$, and the diameter $d$ range over the set (cf.~\eqref{Omega})
\begin{equation*}
	\Omega_{1,0}=\{(\epsilon,\epsilon^*,d)\in\Omega: \epsilon^*\geqslant\epsilon, \ d>0\}.
\end{equation*}
Every $W$ is thin and has a basis $v_i$ $(0\leqslant i\leqslant d)$ such that $v_i\in E_{\epsilon+i}^*W$ for all $i$, and that the matrix representing $A|_W$ with respect to it agrees with the matrix \eqref{standard form} with parameters
\begin{equation*}
	h=\frac{q^{2D+2-\epsilon^*}}{(1-q)^2}, \qquad r=q^{\epsilon^*-D-2}, \qquad s=q^{2\epsilon^*-2D-3}, \qquad \lambda_0=\theta_{\epsilon^*}.
\end{equation*}
The isomorphism classes in $V_{1,0}$ are determined by $\epsilon,\epsilon^*$, and $d$, and the corresponding multiplicity $m_{\epsilon,\epsilon^*,d}^{1,0}$ in $V_{1,0}$ is given by
\begin{align*}
	\frac{m_{\epsilon,\epsilon^*,d}^{1,0}}{ (q)_{D-1} (q)_{D+1} }&=\frac{ (-1)^{D-d} (1-q^{2D-2\epsilon^*-d+2}) (1-q^d) q^{D-2\epsilon+\epsilon^*-d+\binom{2\epsilon-D+d}{2}} }{ (q)_{D-\epsilon} (q)_{2D-\epsilon-\epsilon^*-d+2} (q)_{D-\epsilon-d} (q)_{2\epsilon-D+d} (q)_{\epsilon^*-\epsilon} }.
\end{align*}
\item Similar statements to (i) above hold for $V_{1,1}$ with $r=q^{\epsilon^*-D-1}$, where we replace $\Omega_{1,0}$ and $m_{\epsilon,\epsilon^*,d}^{1,0}$ by $\Omega_{1,1}$ and $m_{\epsilon,\epsilon^*,d}^{1,1}$, respectively, where
\begin{gather*}
	\Omega_{1,1}=\{(\epsilon,\epsilon^*,d)\in\Omega: \epsilon^*\geqslant\epsilon, \ 2\epsilon+d>D\}, \\
	\frac{m_{\epsilon,\epsilon^*,d}^{1,1}}{ (q)_{D-1} (q)_{D+1} } = \frac{ (-1)^{D-d+1} (1-q^{2D-2\epsilon^*-d+1}) (1-q^{d+1}) q^{D-2\epsilon+\epsilon^*-d+1+\binom{2\epsilon-D+d-1}{2}} }{ (q)_{D-\epsilon+2} (q)_{2D-\epsilon-\epsilon^*-d+1} (q)_{D-\epsilon-d+1} (q)_{2\epsilon-D+d-1} (q)_{\epsilon^*-\epsilon} }.
\end{gather*}
\end{enumerate}
\end{theorem}

\begin{proof}
We let $V_{1,0}$ (resp.~$V_{1,1}$) be the sum of the $W_1$ for which the corresponding $\tilde{\mathcal{W}}$ satisfy $\nu+\rho\geqslant \mu$ (resp.~$\nu+\rho<\mu$).
All the computations are routinely done using \eqref{range1}, \eqref{range2}, and Theorem \ref{irreducible tilde-H-modules}\,(i).
\end{proof}

Concerning $W_2$, the endpoint $\epsilon$ and the diameter $d$ are given by
\begin{equation*}
	(\epsilon,d)=\begin{cases} (\nu+\rho+1,D-2\nu-\rho-2) & \text{if} \ \nu+\rho+1\geqslant \mu, \\ (\mu,D-2\mu+\rho+1) & \text{if} \ \nu+\rho+1<\mu. \end{cases}
\end{equation*}
In this case, we consider the matrix \eqref{standard form} with parameters
\begin{equation*}
	h=\frac{q^{2D+1-\nu-\mu}}{(1-q)^2}, \qquad s=q^{2\nu+2\mu-2D-1}, \qquad \lambda_0=\theta_{\nu+\mu+1},
\end{equation*}
and
\begin{equation*}
	r=\begin{cases} q^{\nu+\mu-D-1} & \text{if} \ \nu+\rho+1\geqslant \mu, \\ q^{\nu+\mu-D} & \text{if} \ \nu+\rho+1<\mu. \end{cases}
\end{equation*}
A similar argument shows that the dual endpoint $\epsilon^*$ of $W_2$ is given by
\begin{equation*}
	\epsilon^*=\nu+\mu+1.
\end{equation*}

\begin{theorem}\label{irreducible T-modules of type 2}
Let $V_2$ be the sum of all the $W_2$ obtained as above, where the $\tilde{\mathcal{W}}$ are over the irreducible $\tilde{\mathcal{H}}$-modules in $\mathbb{C}\tilde{P}$ with $\tau=0$.
Then we have
\begin{equation*}
	V_2=V_{2,0}\bigoplus V_{2,1},
\end{equation*}
where $V_{2,0}$ and $V_{2,1}$ are $T$-submodules of $V_2$ such that the following hold:
\begin{enumerate}
\item Similar statements to Theorem \ref{irreducible T-modules of type 1}\,(i) hold for $V_{2,0}$ with $r=q^{\epsilon^*-D-2}$, where we replace $\Omega_{1,0}$ and $m_{\epsilon,\epsilon^*,d}^{1,0}$ by $\Omega_{2,0}$ and $m_{\epsilon,\epsilon^*,d}^{2,0}$, respectively, where
\begin{gather*}
	\Omega_{2,0}=\{(\epsilon,\epsilon^*,d)\in\Omega: \epsilon^*\geqslant\epsilon, \ \epsilon+d<D\}, \\
	\frac{m_{\epsilon,\epsilon^*,d}^{2,0}}{ (q)_{D-1} (q)_{D+1} } = \frac{ (-1)^{D-d} (1-q^{2D-2\epsilon^*-d+2}) (1-q^{d+2}) q^{D-2\epsilon+\epsilon^*-d-1+\binom{2\epsilon-D+d}{2}} }{ (q)_{D-\epsilon+1} (q)_{2D-\epsilon-\epsilon^*-d+2} (q)_{D-\epsilon-d-1} (q)_{2\epsilon-D+d} (q)_{\epsilon^*-\epsilon} }.
\end{gather*}
\item Similar statements to Theorem \ref{irreducible T-modules of type 1}\,(i) hold for $V_{2,1}$ with $r=q^{\epsilon^*-D-1}$, where we replace $\Omega_{1,0}$ and $m_{\epsilon,\epsilon^*,d}^{1,0}$ by $\Omega_{2,1}$ and $m_{\epsilon,\epsilon^*,d}^{2,1}$, respectively, where
\begin{gather*}
	\Omega_{2,1}=\{(\epsilon,\epsilon^*,d)\in\Omega: \epsilon^*>\epsilon, \ 2\epsilon+d>D\}, \\
	\frac{m_{\epsilon,\epsilon^*,d}^{2,1}}{ (q)_{D-1} (q)_{D+1} } = \frac{ (-1)^{D-d+1} (1-q^{2D-2\epsilon^*-d+3}) (1-q^{d+1}) q^{D-2\epsilon+\epsilon^*-d+\binom{2\epsilon-D+d-1}{2}} }{ (q)_{D-\epsilon+2} (q)_{2D-\epsilon-\epsilon^*-d+2} (q)_{D-\epsilon-d+1} (q)_{2\epsilon-D+d-1} (q)_{\epsilon^*-\epsilon-1} }.
\end{gather*}
\end{enumerate}
\end{theorem}

\begin{proof}
We let $V_{2,0}$ (resp.~$V_{2,1}$) be the sum of the $W_2$ for which the corresponding $\tilde{\mathcal{W}}$ satisfy $\nu+\rho+1\geqslant \mu$ (resp.~$\nu+\rho+1<\mu$).
Recall that we have the additional assumption in this case that $\rho<D-2\nu-1$, so that $W_2\ne 0$.
\end{proof}

Next, let $\tilde{\mathcal{W}}$ be an irreducible $\tilde{\mathcal{H}}$-module in $\mathbb{C}\tilde{P}$ with lower endpoint $(\nu,\mu,1)$ and index $\rho$, where we recall from Theorem \ref{irreducible tilde-H-modules} that $\nu,\mu$, and $\rho$ satisfy
\begin{gather*}
	0\leqslant \nu\leqslant D-1, \qquad 0\leqslant \mu\leqslant D+1, \\
	\max\{0,2\mu-D\}\leqslant \rho\leqslant \min\{D-2\nu-1,\mu+1\}.
\end{gather*}
Let the basis vectors $w_{i,j,1}$ of $\tilde{\mathcal{W}}$ be as in Theorem \ref{irreducible tilde-H-modules}\,(ii).
Let
\begin{equation*}
	W_3=\tilde{\mathcal{E}}^*\tilde{\mathcal{W}}=\operatorname{span}\bigl\{ w_{D-i,i,1}: \max\{\nu+\rho+1,\mu\}\leqslant i\leqslant \min\{D-\nu,D-\mu+\rho\} \bigr\}.
\end{equation*}
Then it follows from Theorem \ref{irreducible tilde-H-modules}\,(ii) and \eqref{A11} that
\begin{align*}
	Aw_{D-i,i,1} =& \left\{ q^{\nu+i+1}\gauss{D-\nu-i}{1}\gauss{i-\nu-\rho}{1} -\gauss{D+1}{1} \right. \\
	& \qquad\qquad \left. +\,q^{\mu-\rho+i}\gauss{D-\mu+\rho-i+1}{1}\gauss{i-\mu}{1} \right\} w_{D-i,i,1} \\
	& +q^{2\nu-D+2i+2}\gauss{D-\nu-i}{1}\gauss{i-\nu-\rho}{1} w_{D-i-1,i+1,1} \\
	& +q^{D-\nu+\mu-\rho}\gauss{D-\mu+\rho-i+1}{1}\gauss{i-\mu}{1}w_{D-i+1,i-1,1}
\end{align*}
for all $i$, where we understand that $w_{D-i,i,1}=0$ whenever it is undefined.
We can similarly show that $W_3$ is a thin irreducible $T$-module.
The endpoint $\epsilon$ and the diameter $d$ of $W_3$ are given by
\begin{equation*}
	(\epsilon,d)=\begin{cases} (\nu+\rho+1,D-2\nu-\rho-1) & \text{if} \ \nu+\rho\geqslant \mu, \\ (\mu,D-2\mu+\rho) & \text{if} \ \nu+\rho<\mu. \end{cases}
\end{equation*}
Consider the matrix \eqref{standard form} with parameters
\begin{equation*}
	h=\frac{q^{2D+1-\nu-\mu}}{(1-q)^2}, \qquad s=q^{2\nu+2\mu-2D-1}, \qquad \lambda_0=\theta_{\nu+\mu+1},
\end{equation*}
and
\begin{equation*}
	r=\begin{cases} q^{\nu+\mu-D} & \text{if} \ \nu+\rho\geqslant \mu, \\ q^{\nu+\mu-D-1} & \text{if} \ \nu+\rho<\mu. \end{cases}
\end{equation*}
Then we find that the dual endpoint $\epsilon^*$ of $W_3$ is given by
\begin{equation*}
	\epsilon^*=\nu+\mu+1.
\end{equation*}

\begin{theorem}\label{irreducible T-modules of type 3}
Let $V_3$ be the sum of all the $W_3$ obtained as above, where the $\tilde{\mathcal{W}}$ are over the irreducible $\tilde{\mathcal{H}}$-modules in $\mathbb{C}\tilde{P}$ with $\tau=1$.
Then we have
\begin{equation*}
	V_3=V_{3,0}\bigoplus V_{3,1},
\end{equation*}
where $V_{3,0}$ and $V_{3,1}$ are $T$-submodules of $V_3$ such that the following hold:
\begin{enumerate}
\item Similar statements to Theorem \ref{irreducible T-modules of type 1}\,(i) hold for $V_{3,0}$ with $r=q^{\epsilon^*-D-2}$, where we replace $\Omega_{1,0}$ and $m_{\epsilon,\epsilon^*,d}^{1,0}$ by $\Omega_{3,0}$ and $m_{\epsilon,\epsilon^*,d}^{3,0}$, respectively, where
\begin{gather*}
	\Omega_{3,0}=\{(\epsilon,\epsilon^*,d)\in\Omega: \epsilon^*>\epsilon\}, \\
	\begin{split}
	\frac{m_{\epsilon,\epsilon^*,d}^{3,0}}{ (q)_{D-1} (q)_{D+1} } &= \frac{ (-1)^{D-d+1} (1-q^{2D-2\epsilon^*-d+2}) (1-q^{d+1}) q^{D-2\epsilon+\epsilon^*-d-1+\binom{2\epsilon-D+d}{2}} }{ (q)_{D-\epsilon+2} (q)_{2D-\epsilon-\epsilon^*-d+1} (q)_{D-\epsilon-d+1} (q)_{2\epsilon-D+d} (q)_{\epsilon^*-\epsilon-1} } \\
	& \qquad \times (q^{D+3}-q^{D-\epsilon+2}-q^{D-\epsilon-d+1}+1).
	\end{split}
\end{gather*}
\item Similar statements to Theorem \ref{irreducible T-modules of type 1}\,(i) hold for $V_{3,1}$ with $r=q^{\epsilon^*-D-1}$, where we replace $\Omega_{1,0}$ and $m_{\epsilon,\epsilon^*,d}^{1,0}$ by $\Omega_{3,1}$ and $m_{\epsilon,\epsilon^*,d}^{3,1}$, respectively, where
\begin{gather*}
	\Omega_{3,1}=\{(\epsilon,\epsilon^*,d)\in\Omega: \epsilon^*=\epsilon-1\}\sqcup \{(\epsilon,\epsilon^*,d)\in\Omega: \epsilon^*\geqslant\epsilon, \ 2\epsilon+d>D\}, \\
	\begin{split}
	\frac{m_{\epsilon,\epsilon^*,d}^{3,1}}{ (q)_{D-1} (q)_{D+1} } &= \frac{ (-1)^{D-d} (1-q^{2D-2\epsilon^*-d+2}) (1-q^{d+1}) q^{D-2\epsilon+\epsilon^*-d+\binom{2\epsilon-D+d-1}{2}} }{ (q)_{D-\epsilon+1} (q)_{2D-\epsilon-\epsilon^*-d+3} (q)_{D-\epsilon-d} (q)_{2\epsilon-D+d-1} (q)_{\epsilon^*-\epsilon+1} } \\
	& \qquad \times (q^{D+3}-q^{2D-\epsilon-\epsilon^*-d+3}-q^{\epsilon^*-\epsilon+1}+1).
	\end{split}
\end{gather*}
\end{enumerate}
\end{theorem}

\begin{proof}
In this case, we let $V_{3,0}$ (resp.~$V_{3,1}$) be the sum of the $W_3$ for which the corresponding $\tilde{\mathcal{W}}$ satisfy $\nu+\rho<\mu$ (resp.~$\nu+\rho\geqslant\mu$).
\end{proof}

We have
\begin{equation*}
	\mathbb{C}X=V_{1,0}\bigoplus V_{1,1}\bigoplus V_{2,0}\bigoplus V_{2,1}\bigoplus V_{3,0}\bigoplus V_{3,1},
\end{equation*}
and hence Theorems \ref{irreducible T-modules of type 1}, \ref{irreducible T-modules of type 2}, and \ref{irreducible T-modules of type 3} give all the irreducible $\tilde{\mathcal{H}}$-modules up to isomorphism.
In view of Lemma \ref{uniqueness of r}, it follows that
\begin{itemize}
\item Irreducible $T$-modules in $V_{1,0}\bigoplus V_{2,0}\bigoplus V_{3,0}$ are isomorphic if and only if they have the same $\epsilon,\epsilon^*$, and $d$.
\item Irreducible $T$-modules in $V_{1,1}\bigoplus V_{2,1}\bigoplus V_{3,1}$ are isomorphic if and only if they have the same $\epsilon,\epsilon^*$, and $d$.
\item Irreducible $T$-modules with $d=0$ are isomorphic if and only if they have the same $\epsilon$ and $\epsilon^*$.
\item There are no other isomorphisms.
\end{itemize} 
By these comments, we may compute the multiplicity of an irreducible $T$-module in $\mathbb{C}X$ simply by summing up the corresponding multiplicities in the above summands, but we omit the formulas as these seem too complicated. 

\begin{remark}
It is known that every irreducible $T$-module of the Grassmann graph $J_q(n,D)$ satisfies $\epsilon^*\geqslant \epsilon$; cf.~\cite{LIW2020LAA,TTW2021+pre,Terwilliger1993JACb,Watanabe2015M}.
On the other hand, $\tilde{J}_q(2D+1,D)$ has irreducible $T$-modules with $\epsilon^*=\epsilon-1$.
\end{remark}

\begin{remark}
Using the above results, we can compute the spectrum of the \emph{local graph} of $\tilde{J}_q(2D+1,D)$ with respect to $x$, i.e., the induced subgraph on the neighbors of $x$.
First, there is a unique irreducible $T$-module with $\epsilon=0$, namely, $Tx$.
This module is called the \emph{primary} $T$-module.
It satisfies $\epsilon^*=0$ and $d=D$ (cf.~\cite[Lemma 3.6]{Terwilliger1992JAC}), and resides in $V_{1,0}$.
From the primary $T$-module, we obtain the eigenvalue $\mathsf{a}_1=q(1+q)\gauss{D}{1}-1$ (cf.~\eqref{standard form}) with multiplicity one.
Next, there are four types of irreducible $T$-modules with $\epsilon=1$, namely, those with $(\epsilon^*,d)=(1,D-2)$ in $V_{1,0}\bigoplus V_{2,0}$, those with $(\epsilon^*,d)=(1,D-1)$ in $V_{1,0}$, those with $(\epsilon^*,d)=(2,D-2)$ in $V_{1,0}\bigoplus V_{2,0}\bigoplus V_{3,0}$, and those with $(\epsilon^*,d)=(1,D-1)$ in $V_{1,1}\bigoplus V_{3,1}$.
The corresponding eigenvalues $\mathsf{a}_0$ are $q^2\gauss{D}{1}-1$, $-1$, $-q-1$, and $q^2\gauss{D-1}{1}-1$, respectively.
By computing their multiplicities, we obtain the spectrum as follows:
\begin{equation*}
\left[%
\begin{array}{cccccc}
	q(1+q)\gauss{D}{1}-1 & q^2\gauss{D}{1}-1 & -1 & -q-1 & q^2\gauss{D-1}{1}-1 \\[2mm]
	1 & \gauss{D-1}{1} & (q^{D+1}-1)\gauss{D-1}{1} & q^2\gauss{D-1}{1}\left(\gauss{D+1}{1}-q^{D-1}\right) & q\gauss{D+1}{1}-1
\end{array}%
\right]
\end{equation*}
This spectrum was first found by Bang, Fujisaki, and Koolen \cite[Theorem 1.1\,(ii)]{BFK2009EJC}.
\end{remark}

\section*{Acknowledgments}

TW gratefully acknowledges financial support from Professor Tatsuro Ito so that TW can concentrate on this research.
HT was supported by JSPS KAKENHI Grant Numbers JP17K05156 and JP20K03551.


\end{document}